\documentclass[10pt,a4paper]{article}

\usepackage{latexsym,amsfonts,amsmath,amssymb,amsthm,mathrsfs,color,graphicx}

\newtheorem{theorem}{Theorem}
\newtheorem{lemma}{Lemma}
\newtheorem{algorithm}{Algorithm}
\newtheorem{remark}{Remark}

\newtheorem{definition}{Definition}
\newtheorem{corollary}{Corollary}

\newcommand{\rd}{\,\mathrm{d}}
\newcommand{\rtr}{\,\mathrm{tr}}
\newcommand{\bsx}{\boldsymbol{x}}

\newcommand{\bsz}{\boldsymbol{z}}
\newcommand{\bsl}{\boldsymbol{l}}
\newcommand{\bsk}{\boldsymbol{k}}

\newcommand{\bsq}{\boldsymbol{q}}

\newcommand{\bsgamma}{\boldsymbol{\gamma}}

\newcommand{\bszero}{\boldsymbol{0}}

\newcommand{\nat}{\mathbb{N}}

\newcommand{\RR}{\mathbb{R}}

\newcommand{\FF}{\mathbb{F}}

\newcommand{\Dcal}{\mathcal{D}}
\newcommand{\Ecal}{\mathcal{E}}

\newcommand{\Lcal}{\mathcal{L}}
\newcommand{\Ocal}{\mathcal{O}}

\newcommand{\Rcal}{\mathcal{R}}
\newcommand{\Wcal}{\mathcal{W}}
\newcommand{\wal}{\mathrm{wal}}

\allowdisplaybreaks

\begin{document}

\title{Good interlaced polynomial lattice rules for numerical integration in weighted Walsh spaces}

\author{Takashi Goda\thanks{Graduate School of Engineering, The University of Tokyo, 7-3-1 Hongo, Bunkyo-ku, Tokyo 113-8656 (\tt{goda@frcer.t.u-tokyo.ac.jp})}}

\date{\today}

\maketitle

\begin{abstract}
Quadrature rules using higher order digital nets and sequences are known to exploit the smoothness of a function for numerical integration and to achieve an improved rate of convergence as compared to classical digital nets and sequences for smooth functions. A construction principle of higher order digital nets and sequences based on a digit interlacing function was introduced in [J. Dick, SIAM J. Numer. Anal., 45 (2007) pp.~2141--2176], which interlaces classical digital nets or sequences whose number of components is a multiple of the dimension.

In this paper, we study the use of polynomial lattice point sets for interlaced components. We call quadrature rules using such point sets {\em interlaced polynomial lattice rules}. We consider weighted Walsh spaces containing smooth functions and derive two upper bounds on the worst-case error for interlaced polynomial lattice rules, both of which can be employed as a quality criterion for the construction of interlaced polynomial lattice rules. We investigate the component-by-component construction and the Korobov construction as a means of explicit constructions of good interlaced polynomial lattice rules that achieve the optimal rate of the worst-case error. Through this approach we are able to obtain a good dependence of the worst-case error bounds on the dimension under certain conditions on the weights, while significantly reducing the construction cost as compared to higher order polynomial lattice rules.
\end{abstract}
\section{Introduction}\label{sec:intro}
In this paper, we study the approximation of multivariate integrals of smooth functions defined over the $s$-dimensional unit cube $[0,1)^s$,
  \begin{align*}
     I_s(f)=\int_{[0,1)^s}f(\bsx)\rd\bsx ,
  \end{align*}
by averaging function evaluations at $N$ points $\bsx_0,\ldots,\bsx_{N-1} \in [0,1)^s$
  \begin{align*}
     \hat{I}_{N,s}(f)=\frac{1}{N}\sum_{n=0}^{N-1}f(\bsx_n) .
  \end{align*}
Whereas simple Monte Carlo methods use randomly chosen sample points and achieve the root mean square error of order $N^{-1/2}$, quasi-Monte Carlo (QMC) methods aim at improving the convergence rate of the worst-case error by employing carefully designed deterministic point sets as quadrature points that are distributed as uniformly as possible. The Koksma-Hlawka inequality states that the integration error $|\hat{I}_{N,s}(f)-I_s(f)|$ is bounded above by the star-discrepancy of a point set times the variation of an integrand in the sense of Hardy and Krause, see for example \cite[Chapter~2]{Nie92a}. Since the variation of an integrand is independent of a point set, QMC methods have a deterministic worst-case error whose order equals that of the star-discrepancy and is typically given by $N^{-1+\delta}$ for any $\delta>0$. There are two prominent families for constructing QMC point sets: integration lattices \cite{Nie92a,SJ94} and digital nets and sequences \cite{DP10,Nie92a}. Regarding explicit constructions of classical digital sequences, we refer to \cite[Chapter~8]{DP10} and \cite[Chapter~4]{Nie92a}. Polynomial lattice point sets that were first proposed in \cite{Nie92b} are one of the special constructions for classical digital nets and have been extensively studied in the literature, see for example \cite[Chapter~10]{DP10} and \cite{P12}. We refer to QMC rules using a polynomial lattice point set as polynomial lattice rules.

The drawback of classical digital nets and sequences is that they cannot exploit the smoothness (in the classical, non-digital sense) of a function for numerical integration beyond order 1, and thus, it is not possible for them to achieve a higher convergence rate for smooth integrands. Higher order digital nets and sequences were introduced in \cite{Dic07a,Dic08,Dic09}, and they exploit the smoothness of an integrand and achieve the improved rate of convergence of order $N^{-\alpha+\delta}$ for any $\delta>0$, where $\alpha$ is an integer greater than 1 which measures the smoothness of a function. Recent applications in the area of uncertainty quantification, in particular partial differential equations with random coefficients, are in need of using these types of quadrature rules, see for example \cite{KSS12}.

Two construction principles of higher order digital nets and sequences have been proposed so far. One is known as higher order polynomial lattice rules that are given by generalizing the definition of polynomial lattice rules, see for instance \cite{BDGP11,BDLNP12,DP07}. The other is based on a digit interlacing function applied to classical digital nets and sequences whose number of components is a multiple $ds$ of the dimension, see for instance \cite{Dic07a,Dic08}.

In this study, we focus on the latter construction principle, wherein we use polynomial lattice point sets for interlaced components. We call such point sets {\em interlaced polynomial lattice point sets} and quadrature rules using interlaced polynomial lattice point sets {\em interlaced polynomial lattice rules}. In this context, randomization of interlaced polynomial lattice point sets using digital shift and scrambling has been studied very recently in \cite{Godxx} and \cite{GDxx} respectively. What we are concerned with in this study is to construct {\em deterministic} interlaced polynomial lattice rules that achieve the optimal convergence rate of the worst-case error for numerical integration of smooth functions.

In order to make interlaced polynomial lattice rules available, we need to have a computable quality criterion that enables us to obtain explicit constructions of good polynomial lattice point sets used for interlaced components. In this study, we consider weighted Walsh spaces of smoothness $\alpha$ with general weights and derive two upper bounds on the worst-case error for interlaced polynomial lattice rules. Here we note that the weights model the dependence of an integrand on certain projections as discussed in \cite{SW98}. One bound applies to the cases $d\le \alpha$ and $d> \alpha$, while the other tighter bound applies only to the case $d\le \alpha$. Employing either of these bounds as a quality criterion, we show that the component-by-component (CBC) construction and the Korobov construction can be used to obtain explicit constructions of good polynomial lattice point sets as basis for interlacing construction. Interlaced polynomial lattice rules thus constructed achieve the worst-case error of order $N^{-\min(\alpha,d)+\delta}$. When $d\ge \alpha$, this convergence rate is best possible \cite{Sha63} (apart from the power of the hidden $\log N$ factor). The resulting advantage of interlaced polynomial lattice rules over higher order polynomial lattice rules lies in the significantly reduced construction cost, which is an important aspect from a practical viewpoint. Another resulting advantage of interlaced polynomial lattice rules over the use of digital $(t,m,ds)$-nets or digital $(t,ds)$-sequences for interlaced components as in \cite{Dic07a,Dic08} are the better dependence of the worst-case error on the dimension and the possibility to construct the rules for a given set of weights in the Walsh space (as needed for instance in \cite{KSS12}).

The remainder of this paper is organized as follows. In Section~\ref{sec:pre}, we describe the necessary background and notation. In Section~\ref{sec:error}, we introduce the weighted Walsh space of smoothness $\alpha$ as in \cite{Dic08} and derive two upper bounds on the worst-case error for interlaced polynomial lattice rules. In Sections~\ref{sec:cbc} and~\ref{sec:korobov}, we investigate the CBC construction and Korobov construction respectively, and discuss the dependence of the worst-case error on the dimension for each construction.

\section{Preliminaries}\label{sec:pre}

We use the following notation. Let $\nat$ be the set of positive integers and $\nat_0:=\nat \cup \{0\}$. Given a prime $b$, let $\FF_b:=\{0,1,\ldots,b-1\}$ be the finite field consisting of $b$ elements. We identify the elements of $\FF_b$ with the set of integers $\{0,1,\ldots,b-1\}$. For $a,c\in \nat$ such that $a\le c$ we denote by $\{a:c\}$ the index set $\{a,a+1,\ldots,c-1,c\}$.

\subsection{Polynomial lattice rules}
Given a prime $b$, let us denote by $\FF_b((x^{-1}))$ the field of formal Laurent series over $\FF_b$. Every element of $\FF_b((x^{-1}))$ has the form
  \begin{align*}
    L = \sum_{l=w}^{\infty}t_l x^{-l} ,
  \end{align*}
where $w$ is an arbitrary integer and all $t_l\in \FF_b$. Furthermore, we denote by $\FF_b[x]$ the set of all polynomials over $\FF_b$. For a given integer $m$, we define the mapping $v_m$ from $\FF_b((x^{-1}))$ to the interval $[0,1)$ by
  \begin{align*}
    v_m\left( \sum_{l=w}^{\infty}t_l x^{-l}\right) =\sum_{l=\max(1,w)}^{m}t_l b^{-l}.
  \end{align*}
A non-negative integer $k$ whose $b$-adic expansion is given by $k=\kappa_0+\kappa_1 b+\cdots +\kappa_{a-1} b^{a-1}$ will be identified with the polynomial $k(x)=\kappa_0+\kappa_1 x+\cdots +\kappa_{a-1} x^{a-1}\in \FF_b[x]$.  For $\bsk=(k_1,\ldots, k_s)\in (\FF_b[x])^s$ and $\bsq=(q_1,\ldots, q_s)\in (\FF_b[x])^s$, we define the inner product as
  \begin{align}\label{eq:inner_product}
     \bsk \cdot \bsq := \sum_{j=1}^{s}k_j q_j \in \FF_b[x] ,
  \end{align}
and we write $q\equiv 0 \pmod p$ if $p$ divides $q$ in $\FF_b[x]$. Using this notation, a polynomial lattice point set is constructed as follows.

\begin{definition}\label{def:polynomial_lattice}(Polynomial lattice rules)
Let $m, s \in \nat$. Let $p \in \FF_b[x]$ such that $\deg(p)=m$ and let $\bsq=(q_1,\ldots,q_s) \in (\FF_b[x])^s$. A polynomial lattice point set $P_{b^m,s}(\bsq,p)$ is a point set consisting of $b^m$ points $\bsx_0,\ldots,\bsx_{b^m-1}$ that are defined as
  \begin{align*}
    \bsx_n &:= \left( v_m\left( \frac{n(x)q_1(x)}{p(x)} \right) , \ldots , v_m\left( \frac{n(x)q_s(x)}{p(x)} \right) \right) \in [0,1)^s ,
  \end{align*}
for $0\le n<b^m$. A QMC rule using this point set is called a \emph{polynomial lattice rule} with generating vector $\bsq$ and modulus $p$.
\end{definition}

In the remainder of this paper, we denote by $P_{b^m,s}(\bsq,p)$ a polynomial lattice point set, implicitly meaning that $\deg(p)=m$ and the number of components in the vector $\bsq$ is $s$.

We add one more notation and introduce the concept of the so-called {\it dual polynomial lattice} of a polynomial lattice point set. For $k\in \nat_0$ with $b$-adic expansion $k= \kappa_0 + \kappa_1 b+\cdots + \kappa_{a-1} b^{a-1}$, let $\rtr_m(k)$ be the polynomial of degree less than $m$ obtained by truncating the associated polynomial $k(x)\in \FF_b[x]$ as
  \begin{align*}
    \rtr_m(k)= \kappa_0 + \kappa_1 x+\cdots + \kappa_{m-1}x^{m-1},
  \end{align*}
where we set $\kappa_{a} = \cdots = \kappa_{m-1} = 0$ if $a< m$. For a vector $\bsk=(k_1,\ldots, k_s)\in \nat_0^s$, we define $\rtr_m(\bsk)=(\rtr_m(k_1),\ldots, \rtr_m(k_s))$. Then the dual polynomial lattice is defined as follows.

\begin{definition}\label{def:dual_net}
Let $P_{b^m,s}(\bsq,p)$ be a polynomial lattice point set. Then the dual polynomial lattice of $P_{b^m,s}(\bsq,p)$ is defined as
  \begin{align*}
     D^\perp_{\bsq,p} := \{ \bsk=(k_1,\ldots,k_s)\in \nat_0^{s}:\ \rtr_m(\bsk) \cdot \bsq\equiv 0 \pmod p \} .
  \end{align*}
where the inner product is defined in the sense of (\ref{eq:inner_product}).
\end{definition}

\subsection{Higher order digital nets}\label{ssec:ho_digital_nets}

Higher order digital nets exploit the smoothness of an integrand so that they achieve the optimal order of convergence of the deterministic worst-case error for functions with smoothness $\alpha>1$. The result is based on a bound on the decay of the Walsh coefficients of smooth functions, see \cite{Dic08}. We refer to \cite{Dic09} for a brief introduction of the central ideas.

As mentioned in the introduction, there exist two construction principles of higher order digital nets and sequences so far. One is called higher order polynomial lattice rules that are defined as follows. In Definition~\ref{def:polynomial_lattice}, we set $p$ with $\deg(p)=m'>m$ and replace $v_m$ with $v_{m'}$ for the mapping function. Then a higher order polynomial lattice point set consists of the first $b^m$ points of a classical polynomial lattice point set with $b^{m'}$ points (where $m'$ is recommended to equal $\alpha m$ for integrands of smoothness $\alpha$). To summarize, a higher order polynomial lattice point set is constructed as follows.

\begin{definition}\label{def:ho_polynomial_lattice}(Higher order polynomial lattice rules)
Let $m,m', s \in \nat$ be such that $m'> m$. Let $p \in \FF_b[x]$ such that $\deg(p)=m'$ and let $\bsq=(q_1,\ldots,q_s) \in (\FF_b[x])^s$. A higher order polynomial lattice point set is a point set consisting of $b^m$ points $\bsx_0,\ldots,\bsx_{b^m-1}$ that are defined as
  \begin{align*}
    \bsx_n &:= \left( v_{m'}\left( \frac{n(x)q_1(x)}{p(x)} \right) , \ldots , v_{m'}\left( \frac{n(x)q_s(x)}{p(x)} \right) \right) \in [0,1)^s ,
  \end{align*}
for $0\le n<b^m$. A QMC rule using this point set is called a \emph{higher order polynomial lattice rule} with generating vector $\bsq$ and modulus $p$.
\end{definition}

The existence of higher order polynomial lattice rules achieving the optimal order of convergence was established in \cite{DP07} and the CBC construction was proved to achieve the optimal order of convergence in \cite{BDGP11}. Using an efficient calculation of the worst-case error in the weighted Walsh space with product weights as implemented in \cite{BDLNP12}, the computational cost of $\Ocal(sb^{m'}\log b^{m'})$ operations using $\Ocal(b^{m'})$ memory is required for the CBC construction to find a good set of polynomials $\bsq=(q_1,\ldots,q_s)$. Thus the computational cost depends exponentially on $\alpha$ when $m' = \alpha m$.

The other construction principle of higher order digital nets is based on a digit interlacing function applied to classical digital nets and sequences whose number of components is $ds$, where $d$ is an integer greater than 1, which is called the interlacing factor. In the digit interlacing approach, we first construct a point set consisting of $b^m$ points in $[0,1)^{ds}$ instead of $[0,1)^{s}$. Here we denote each point in $[0,1)^{ds}$ by $\bsz_n=(z_{n,1},\ldots, z_{n,ds})$ for $0\le n<b^m$. Then every consecutive $d$ components of a point $\bsz_n$ are digitally interlaced according to the following digit interlacing function of order $d$ for real numbers
  \begin{align*}
     x_{n,j} = \Dcal_d(z_{n,(j-1)d+1},\ldots,z_{n,jd}) := \sum_{a=1}^{\infty}\sum_{r=1}^{d}z_{(j-1)d+r,n,a}b^{-r-(a-1)d} ,
  \end{align*}
for $1\le j\le s$ and $0\le n<b^m$, where we denote the $b$-adic expansion of $z_{n,j}$ by $z_{n,j}=z_{1,n,j}b^{-1}+z_{2,n,j}b^{-2}+\cdots$ for $1\le j\le ds$, and where we assume that the expansion of $z_{n,j}$ is unique in the sense that infinitely many digits are different from $b-1$. In this way, the $n$-th point $\bsx_n=(x_{n,1},\ldots,x_{n,s})\in [0,1)^{s}$ is obtained. Thus, we have
  \begin{align*}
     \bsx_n = (\Dcal_d(z_{n,1},\ldots,z_{n,d}),\Dcal_d(z_{n,d+1},\ldots,z_{n,2d}),\ldots,\Dcal_d(z_{n,(s-1)d+1},\ldots,z_{n,sd})) .
  \end{align*}
In the following, we simply write $\bsx_n=\Dcal_d(\bsz_n)$ when $\bsx_n$ is obtained by digitally interlacing every $d$ components of $\bsz_n$.

In this paper, we are concerned with the use of polynomial lattice point sets to generate a point set in $[0,1)^{ds}$ that are used as interlaced components for higher order digital nets. We call QMC rules using such point sets {\em interlaced polynomial lattice rules (of order $d$)}. For clarity, we give the definition of interlaced polynomial lattice rules below.

\begin{definition}\label{def:interlacing_polynomial_lattice}(Interlaced polynomial lattice rules)
Let $m,s,d \in \nat$, $d>1$. Let $p \in \FF_b[x]$ such that $\deg(p)=m$ and let $\bsq=(q_1,\ldots,q_{ds}) \in (\FF_b[x])^{ds}$. An interlaced polynomial lattice point set of order $d$ is a point set consisting of $b^m$ points $\bsx_0,\ldots,\bsx_{b^m-1}$ that are defined as
  \begin{align*}
    \bsx_n := \Dcal_d(\bsz_n) ,
  \end{align*}
where the point $\bsz_n$ is the $n$-th point of a polynomial lattice point set $P_{b^m,ds}(\bsq,p)$ which is given as
  \begin{align*}
    \bsz_n := \left( v_m\left( \frac{n(x) \, q_1(x)}{p(x)} \right) , \ldots , v_m\left( \frac{n(x) \, q_{ds}(x)}{p(x)} \right) \right)\in [0,1)^{ds} ,
  \end{align*}
for $0 \le n < b^m$. A QMC rule using this point set is called an \emph{interlaced polynomial lattice rule (of order $d$)} with generating vector $\bsq$ and modulus $p$.
\end{definition}

By Definition \ref{def:interlacing_polynomial_lattice}, in order to construct good interlaced polynomial lattice rules, one needs to find suitable generating vectors $\bsq=(q_1,\ldots,q_{ds})$. In the following, we discuss computer search algorithms which find good generating vectors.

In the remainder of this paper, we simply write $\Dcal_d(P_{b^m,ds}(\bsq,p))$ for an interlaced polynomial lattice point set of order $d$, implicitly meaning that $\deg(p)=m$ and the number of components in the vector $\bsq$ is $ds$. In order to introduce the dual polynomial lattice of $\Dcal_d(P_{b^m,ds}(\bsq,p))$, we need one more notation.

For an interlacing factor $d>1$ and $k_1,\ldots,k_d\in \nat_0^d$, we denote the $b$-adic expansion of $k_j$ by $k_j=\kappa_{j,0}+\kappa_{j,1}b+\cdots$ for $1\le j\le d$. Then the digit interlacing function of order $d$ for non-negative integers is defined as
  \begin{align}\label{eq:integer_interlace}
    \Ecal_d(k_1,\ldots, k_d) := \sum_{a=0}^{\infty}\sum_{r=1}^{d}\kappa_{r,a}b^{ad+r-1} .
  \end{align}
We extend this function to vectors consisting of $ds$ components as
  \begin{align*}
    \Ecal_d(k_1,\ldots, k_{ds}) = (\Ecal_d(k_1,\ldots, k_d),\Ecal_d(k_{d+1},\ldots, k_{2d}),\ldots, \Ecal_d(k_{(s-1)d+1},\ldots, k_{sd})).
  \end{align*}
Then the dual polynomial lattice of $\Dcal_d(P_{b^m,ds}(\bsq,p))$ is defined as follows.

\begin{definition}\label{def:dual_net2}
Let $\Dcal_d(P_{b^m,ds}(\bsq,p))$ be an interlaced polynomial lattice point set of order $d$. Then the dual polynomial lattice of $\Dcal_d(P_{b^m,ds}(\bsq,p))$ is defined as
  \begin{align*}
     D^\perp_{\bsq,p,d} := \{ \Ecal_d(\bsk):\ \bsk=(k_1,\ldots, k_{ds})\in D^\perp_{\bsq,p}  \}
  \end{align*}
where $D^\perp_{\bsq,p}$ is the dual polynomial lattice of $P_{b^m,ds}(\bsq,p)$ as given in Definition~\ref{def:dual_net}, in which $s$ is replaced with $ds$.
\end{definition}

\section{Numerical integration in weighted Walsh spaces}\label{sec:error}
In this section, we first recall Walsh functions and then introduce the weighted Walsh space of smoothness $\alpha$ as in \cite{Dic08}, which contains all functions whose partial mixed derivatives up to $\alpha$ in each variable are square integrable, where $\alpha$ is an integer greater than 1. Next, we derive two upper bounds on the worst-case error for interlaced polynomial lattice rules of order $d$. The first bound covers both the cases $d\le \alpha$ and $d> \alpha$, while the second tighter bound applies only to the case $d\le \alpha$.

\subsection{Walsh functions}
In order to introduce the weighted Walsh space of smoothness $\alpha$, we need to recall the definition of Walsh functions. Walsh functions were first introduced in \cite{Wal23} for the case of base 2 and were generalized later, see for example \cite{Che55}. We refer to \cite[Appendix~A]{DP10} for more information on Walsh functions in the context of numerical integration. We first give the definition for the one-dimensional case. Recall that every $x\in [0,1)$ has a $b$-adic expansion $x=x_1 b^{-1} + x_2 b^{-2} + \cdots$ with $x_1,x_2,\ldots \in \{0,1,\dots,b-1\}$ and that the coefficients $x_1,x_2,\ldots$ are uniquely determined if we require that infinitely many of them are different from $b-1$. In this paper, we will always assume this unique expansion.

\begin{definition}
Let $b$ be a prime and $\omega_b=e^{2\pi \mathrm{i}/b}$. We denote the $b$-adic expansion of $k\in \nat_0$ by $k = \kappa_0+\kappa_1 b+\cdots +\kappa_{a-1}b^{a-1}$ with $\kappa_i\in \FF_b$. Then the $k$-th $b$-adic Walsh function ${}_b\wal_k: [0,1)\to \{1,\omega_b,\ldots, \omega_b^{b-1}\}$ is defined as
  \begin{align*}
    {}_b\wal_k(x) := \omega_b^{x_1\kappa_0+\cdots+x_a\kappa_{a-1}} ,
  \end{align*}
for $x\in [0,1)$ with its unique $b$-adic expansion $x=x_1 b^{-1} + x_2 b^{-2} + \cdots$.
\end{definition}

This definition can be generalized to the higher-dimensional case.

\begin{definition}
For a dimension $s\in \nat$, let $\bsx=(x_1,\ldots, x_s)\in [0,1)^s$ and $\bsk=(k_1,\ldots, k_s)\in \nat_0^s$. Then the $\bsk$-th $b$-adic Walsh function ${}_b\wal_{\bsk}: [0,1)^s \to \{1,\omega_b,\ldots, \omega_b^{b-1}\}$ is defined as
  \begin{align*}
    {}_b\wal_{\bsk}(\bsx) := \prod_{j=1}^s {}_b\wal_{k_j}(x_j) .
  \end{align*}
\end{definition}

We note that the system $\{{}_b\wal_{\bsk}:\; \bsk\in \nat_0^s\}$ is a complete orthonormal system in $\Lcal_2([0,1)^s)$, see \cite[Theorem~A.11]{DP10}. Since we shall always use Walsh functions in a fixed prime base $b$, we omit the subscript and simply write $\wal_k$ or $\wal_{\bsk}$ in the remainder of this paper. Here we introduce the important connection between an interlaced polynomial lattice point set and Walsh functions.

\begin{lemma}\label{lemma:dual_net_Walsh}
Let $\Dcal_d(P_{b^m,ds}(\bsq,p))=\{\bsx_0,\ldots, \bsx_{b^m-1}\}$ be an interlaced polynomial lattice point set of order $d$ and $D^\perp_{\bsq,p,d}$ be its dual polynomial lattice as defined in Definition~\ref{def:dual_net2}. Then we have
  \begin{align*}
     \frac{1}{b^m}\sum_{n=0}^{b^m-1}\wal_{\bsl}(\bsx_n) = \left\{ \begin{array}{ll}
     1 & \mathrm{if}\ \bsl\in D^\perp_{\bsq,p,d} , \\
     0 & \mathrm{otherwise} . \\
     \end{array} \right.
  \end{align*}
\end{lemma}

\begin{proof}
From the definition of $\Ecal_d$, for any $\bsl\in \nat_0^s$ we have a unique $\bsk\in \nat_0^{ds}$ such that $\bsl=\Ecal_d(\bsk)$. Let us recall that $\bsx_n=\Dcal_d(\bsz_n)$ for $0\le n<b^m$. Then we have
  \begin{align*}
     \frac{1}{b^m}\sum_{n=0}^{b^m-1}\wal_{\bsl}(\bsx_n) & = \frac{1}{b^m}\sum_{n=0}^{b^m-1}\wal_{\Ecal_d(\bsk)}(\Dcal_d(\bsz_n)) \\
     & = \frac{1}{b^m}\sum_{n=0}^{b^m-1}\wal_{\bsk}(\bsz_n) ,
  \end{align*}
where the second equality stems from the definitions of $\Dcal_d$, $\Ecal_d$ and the Walsh functions. Moreover, since $\{\bsz_0,\ldots, \bsz_{b^m-1}\}$ is a polynomial lattice point set, the last expression equals 1 if $\bsk\in D^\perp_{\bsq,p}$ and 0 otherwise, which is straightforward by combining Definition~\ref{def:dual_net}, \cite[Lemma~10.6]{DP10} and \cite[Lemma~4.75]{DP10}. Considering Definition~\ref{def:dual_net2}, the result follows.
\end{proof}

\subsection{Weighted Walsh space of smoothness $\alpha$}
Let us consider the Walsh series of an integrand $f\in \Lcal_2([0,1)^s)$. Because of the orthonormal property of Walsh functions, we have
  \begin{align*}
     f(\bsx)\sim \sum_{\bsl\in \nat_0^{s}}\hat{f}(\bsl)\wal_{\bsl}(\bsx) ,
  \end{align*}
where $\hat{f}(\bsl)$ is the $\bsl$-th Walsh coefficient given by
  \begin{align*}
     \hat{f}(\bsl)=\int_{[0,1)^s}f(\bsx)\overline{\wal_{\bsl}(\bsx)}\rd\bsx .
  \end{align*}
We refer to \cite[Appendix~A.3]{DP10} for a discussion on the pointwise absolute convergence of the Walsh series. In fact, for any function $f:[0,1)^s\to \RR$ in the weighted Walsh space of smoothness $\alpha$ which we shall consider in this paper, its Walsh series converges to $f$ pointwise absolutely.

For functions with square integrable partial mixed derivatives up to $\alpha$ in each variable, upper bounds on the Walsh coefficients were previously obtained in \cite{Dic08}. This result motivates the introduction of the weighted Walsh space of smoothness $\alpha$, which we denote by $\Wcal_{s,\alpha,\bsgamma}$. As shown in \cite{Dic08}, the space $\Wcal_{s,\alpha,\bsgamma}$ contains certain weighted Sobolev spaces of smoothness of order $\alpha$.

We denote the $b$-adic expansion of $k\in \nat$ by $k=\kappa_1 b^{a_1-1}+\cdots+\kappa_\nu b^{a_\nu-1}$ such that $\nu\ge 1$, $1\le a_\nu < \cdots < a_1$ and $\kappa_1,\ldots,\kappa_\nu \in \{1,\ldots,b-1\}$. Given an integer $\alpha\ge 1$, we define
  \begin{align*}
     r_\alpha (k) := b^{-\mu_\alpha(k)} ,
  \end{align*}
where $\mu_\alpha(k)$ is given by
  \begin{align}\label{eq:dick_weight}
     \mu_\alpha(k) := a_1+\cdots +a_{\min(\nu,\alpha)}.
  \end{align}
Furthermore, we define $\mu_\alpha(\bsk_v):=\sum_{j\in v} \mu_\alpha(k_j)$ and $r_\alpha(\bsk_v):=\prod_{j\in v} r_\alpha(k_j)$ for $v\subseteq \{1:s\}$ and $\bsk_v\in \nat^{|v|}$. We put $\mu_\alpha(0)=0$ and thus $r_\alpha (0)=1$. Then the weighted Walsh space $\Wcal_{s,\alpha,\bsgamma}$ with general weights $\bsgamma=(\gamma_v)_{v\subseteq \{1:s\}}$ is defined as follows.

\begin{definition}\label{def:walsh_space}
For smoothness $\alpha>1$ and a set of weights $\bsgamma=(\gamma_v)_{v\subseteq \{1:s\}}$, the weighted Walsh space $\Wcal_{s,\alpha,\bsgamma}$ is a function space consisting of functions $f:[0,1)^s\to \RR$ for which the norm
  \begin{align*}
     \| f\|_{\Wcal_{s,\alpha,\bsgamma}}:=\max_{\substack{v\subseteq \{1:s\} \\ \gamma_v\ne 0}}\gamma_v^{-1}\sup_{\bsl_v\in \nat^{|v|}}\frac{|\hat{f}(\bsl_v,\bszero)|}{r_\alpha(\bsl_v)} 
  \end{align*}
is finite, where $(\bsl_v,\bszero)$ is the vector from $\nat_0^s$ with all the components whose indices are not in $v$ equal $0$, and $|v|$ denotes the cardinality of $v$.
\end{definition}

In this study, we are interested in the worst-case error for numerical integration in $\Wcal_{s,\alpha,\bsgamma}$ using an interlaced polynomial lattice point set $\Dcal_d(P_{b^m,ds}(\bsq,p))$. The initial error for numerical integration in $\Wcal_{s,\alpha,\bsgamma}$ is defined by
  \begin{align*}
     e(P_{0,s},\Wcal_{s,\alpha,\bsgamma}) := \sup_{\substack{f\in \Wcal_{s,\alpha,\bsgamma}\\ \| f\|_{\Wcal_{s,\alpha,\bsgamma}}\le 1}}|I_s(f)| ,
  \end{align*}
where $P_{0,s}$ denotes the empty point set. The worst-case error for numerical integration in $\Wcal_{s,\alpha,\bsgamma}$ using a point set $P_{N,s}\subset [0,1)^s$ is defined by
  \begin{align*}
     e(P_{N,s},\Wcal_{s,\alpha,\bsgamma}) := \sup_{\substack{f\in \Wcal_{s,\alpha,\bsgamma}\\ \| f\|_{\Wcal_{s,\alpha,\bsgamma}}\le 1}}|\hat{I}_{N,s}(f)-I_s(f)| .
  \end{align*}
Applying Definition~\ref{def:dual_net2} and Lemma~\ref{lemma:dual_net_Walsh} to \cite[Theorem~5.1]{Dic08}, we have the following.

\begin{theorem}\label{theorem:worst-case_error}
Let $\Dcal_d(P_{b^m,ds}(\bsq,p))$ be an interlaced polynomial lattice point set of order $d$ and $D^\perp_{\bsq,p,d}$ be its dual polynomial lattice as defined in Definition \ref{def:dual_net2}. Then the initial error for numerical integration in $\Wcal_{s,\alpha,\bsgamma}$ is given by
  \begin{align*}
     e(P_{0,s},\Wcal_{s,\alpha,\bsgamma}) = \gamma_{\emptyset} .
  \end{align*}
The worst-case error for numerical integration in $\Wcal_{s,\alpha,\bsgamma}$ using $\Dcal_d(P_{b^m,ds}(\bsq,p))$ is given by
  \begin{align*}
     e(\Dcal_d(P_{b^m,ds}(\bsq,p)),\Wcal_{s,\alpha,\bsgamma}) = \sum_{\emptyset \ne v\subseteq \{1:s\}}\gamma_{v}\sum_{\substack{\bsl_v\in \nat^{|v|}\\ (\bsl_v,\bszero)\in D^\perp_{\bsq,p,d}}}r_\alpha(\bsl_v).
  \end{align*}
\end{theorem}

From Definition~\ref{def:dual_net2}, we obtain the following corollary of Theorem~\ref{theorem:worst-case_error}. It rewrites the worst-case error in terms of the dual polynomial lattice not of $\Dcal_d(P_{b^m,ds}(\bsq,p))$ but of $P_{b^m,ds}(\bsq,p)$.

\begin{corollary}\label{corollary:worst-case_error}
Let $\Dcal_d(P_{b^m,ds}(\bsq,p))$ be an interlaced polynomial lattice point set of order $d$ and $D^\perp_{\bsq,p}$ be the dual polynomial lattice of $P_{b^m,ds}(\bsq,p)$. The worst-case error for numerical integration in $\Wcal_{s,\alpha,\bsgamma}$ using $\Dcal_d(P_{b^m,ds}(\bsq,p))$ is given by
  \begin{align}\label{eq:worst_case_error}
     e(\Dcal_d(P_{b^m,ds}(\bsq,p)),\Wcal_{s,\alpha,\bsgamma}) = \sum_{\emptyset \ne u\subseteq \{1:ds\}}\gamma_{w(u)}\sum_{\substack{\bsk_u\in \nat^{|u|}\\ (\bsk_u,\bszero)\in D^\perp_{\bsq,p}}}r_\alpha(\Ecal_d(\bsk_u,\bszero)) ,
  \end{align}
where $(\bsk_u,\bszero)$ is the vector from $\nat_0^{ds}$ with all the components whose indices are not in $u$ equal $0$, and we denote by $w(u)$ the set of $1\le j\le s$ such that $u\cap \{(j-1)d+1:jd\}\ne \emptyset$.
\end{corollary}

\begin{proof}
According to Definition~\ref{def:dual_net2}, for any $\bsl\in D^\perp_{\bsq,p,d}(\subseteq \nat_0^s)$ we have a unique $\bsk\in D^\perp_{\bsq,p}(\subseteq \nat_0^{ds})$ such that $\bsl=\Ecal_d(\bsk)$. In the following, we denote by $\hat{\bsk}_v$ the vector of $d|v|$ non-negative integers $(k_{(j-1)d+1},\ldots, k_{jd})_{j\in v}$. Then we have
  \begin{align*}
     e(\Dcal_d(P_{b^m,ds}(\bsq,p)),\Wcal_{s,\alpha,\bsgamma}) & = \sum_{\emptyset \ne v\subseteq \{1:s\}}\gamma_{v}\sum_{\substack{\hat{\bsk}_v\in (\nat_0^{d}\setminus \{\bszero\})^{|v|}\\ (\hat{\bsk}_v,\bszero)\in D^\perp_{\bsq,p}}}r_\alpha(\Ecal_d(\hat{\bsk}_v)) \\
     & = \sum_{\emptyset \ne v\subseteq \{1:s\}}\gamma_{v}\sum_{\substack{\emptyset \ne u \subseteq \{1:ds\} \\ w(u)=v}}\sum_{\substack{\bsk_u\in \nat^{|u|} \\ (\bsk_u,\bszero)\in D^\perp_{\bsq,p}}}r_\alpha(\Ecal_d(\bsk_u,\bszero)) .
  \end{align*}
Swapping the order of sums, the result follows.
\end{proof}

\subsection{A bound on the worst-case error}

As shown in Corollary~\ref{corollary:worst-case_error}, the worst-case error can be represented in terms of the dual polynomial lattice of $P_{b^m,ds}(\bsq,p)$. However, it is hard to give a concise formula for the term $r_\alpha(\Ecal_d(\bsk_u,\bszero))$ in (\ref{eq:worst_case_error}), and thus, $e(\Dcal_d(P_{b^m,ds}(\bsq,p)),\Wcal_{s,\alpha,\bsgamma})$ cannot be employed as a quality criterion for searching for a good set of polynomials $\bsq$. As a remedy, we derive an upper bound on $e(\Dcal_d(P_{b^m,ds}(\bsq,p)),\Wcal_{s,\alpha,\bsgamma})$ for which a concise formula can be provided.

In order to derive an upper bound on the worst-case error, which covers the cases $d\le \alpha$ and $d>\alpha$, we shall use the following lemma from \cite[Lemma~4]{Godxx} about the lower bound on $\mu_\alpha(\Ecal_d(\bsk_u,\bszero))$.

\begin{lemma}\label{lemma:weight}
For $\emptyset \ne u\subseteq \{1:ds\}$ and $\bsk_u\in \nat^{|u|}$, we have
  \begin{align*}
    \mu_{\alpha}(\Ecal_d(\bsk_u,\bszero)) \ge \min(\alpha,d)\sum_{j\in u}\mu_1(k_j)+\frac{1}{2}\alpha|u|-\frac{1}{2}\alpha(2d-1)|w(u)| ,
  \end{align*}
where $\mu_1(k)$ is defined in (\ref{eq:dick_weight}).
\end{lemma}

Let us introduce the following notation
  \begin{align*}
     \tilde{r}_{\alpha,d,(1)} (k) = \left\{ \begin{array}{ll}
     1                                    & \text{if}\; k=0, \\
     b^{-\min(\alpha,d)\mu_1(k)-\alpha/2} & \text{otherwise} .
     \end{array} \right.
  \end{align*}
Further we write $\tilde{r}_{\alpha,d,(1)}(\bsk_u)=\prod_{j\in u}\tilde{r}_{\alpha,d,(1)}(k_j)$. Then we have the following.

\begin{theorem}\label{theorem:bound1}
Let $\Dcal_d(P_{b^m,ds}(\bsq,p))$ be an interlaced polynomial lattice point set of order $d$ and $D^\perp_{\bsq,p}$ be the dual polynomial lattice of $P_{b^m,ds}(\bsq,p)$. The worst-case error for numerical integration in $\Wcal_{s,\alpha,\bsgamma}$ using $\Dcal_d(P_{b^m,ds}(\bsq,p))$ is bounded as
  \begin{align*}
     e(\Dcal_d(P_{b^m,ds}(\bsq,p)),\Wcal_{s,\alpha,\bsgamma}) \le \sum_{\emptyset \ne u\subseteq \{1:ds\}}\tilde{\gamma}_{w(u)}\sum_{\substack{\bsk_u\in \nat^{|u|}\\ (\bsk_u,\bszero)\in D^{\perp}(\bsq,p)}}\tilde{r}_{\alpha,d,(1)}(\bsk_u) ,
  \end{align*}
where we have defined
  \begin{align*}
     \tilde{\gamma}_{w(u)}:=\gamma_{w(u)}b^{\alpha(2d-1)|w(u)|/2} .
  \end{align*}
\end{theorem}

\begin{proof}
Applying Lemma~\ref{lemma:weight} to Corollary~\ref{corollary:worst-case_error}, we have
  \begin{align*}
     & e(\Dcal_d(P_{b^m,ds}(\bsq,p)),\Wcal_{s,\alpha,\bsgamma}) \\
     = & \sum_{\emptyset \ne u\subseteq \{1:ds\}}\gamma_{w(u)}\sum_{\substack{\bsk_u\in \nat^{|u|}\\ (\bsk_u,\bszero)\in D^\perp_{\bsq,p}}}b^{-\mu_\alpha(\Ecal_d(\bsk_u,\bszero))} \\
     \le & \sum_{\emptyset \ne u\subseteq \{1:ds\}}\gamma_{w(u)}b^{\alpha(2d-1)|w(u)|/2}\sum_{\substack{\bsk_u\in \nat^{|u|}\\ (\bsk_u,\bszero)\in D^\perp_{\bsq,p}}}b^{-\sum_{j\in u}(\min(\alpha,d)\mu_1(k_j)+\alpha/2)} .
  \end{align*}
From the above definition of $\tilde{r}_{\alpha,d,(1)} (k)$, the result follows.
\end{proof}

As we shall show below, the upper bound given in Theorem~\ref{theorem:bound1} has a concise formula, so that we can employ this bound as a quality criterion for searching for a good set of polynomials $\bsq$. For simplicity, we denote this bound by
  \begin{align}\label{eq:criterion}
     B_{\alpha,d,\bsgamma,(1)}(\bsq,p) := \sum_{\emptyset \ne u\subseteq \{1:ds\}}\tilde{\gamma}_{w(u)}\sum_{\substack{\bsk_u\in \nat^{|u|}\\ (\bsk_u,\bszero)\in D^{\perp}(\bsq,p)}}\tilde{r}_{\alpha,d,(1)}(\bsk_u) .
  \end{align}

\begin{corollary}\label{corollary:criterion}
Let $\Dcal_d(P_{b^m,ds}(\bsq,p))$ be an interlaced polynomial lattice point set of order $d$ and $P_{b^m,ds}(\bsq,p)=\{\bsz_0,\ldots,\bsz_{b^m-1}\}$ be a polynomial lattice point set with generating vector $\bsq$ and modulus $p$. We have a concise formula for $B_{\alpha,d,\bsgamma,(1)}(\bsq,p)$, which is given by
  \begin{align*}
     B_{\alpha,d,\bsgamma,(1)}(\bsq,p) = \frac{1}{b^m}\sum_{n=0}^{b^m-1}\sum_{\emptyset \ne v\subseteq \{1:s\}}\tilde{\gamma}_{v}\prod_{j\in v}\left[ -1+\prod_{l=1}^{d}\left( 1+\phi_{\alpha,d,(1)}(z_{n,(j-1)d+l})\right)\right] ,
  \end{align*}
where $\tilde{\gamma}_v=\gamma_v b^{\alpha(2d-1)|v|/2}$ and where
  \begin{align*}
     \phi_{\alpha,d,(1)}(z)=\frac{b-1-b^{(\min(\alpha,d)-1)\lfloor \log_b z\rfloor}(b^{\min(\alpha,d)}-1)}{b^{(\alpha+2)/2}(b^{\min(\alpha,d)-1}-1)} ,
  \end{align*}
for any $z\in [0,1)$ in which we set $b^{\lfloor \log_b 0\rfloor}=0$.
\end{corollary}

\begin{proof}
As in the proof of Lemma~\ref{lemma:dual_net_Walsh}, for any $\emptyset \ne u\subseteq \{1:ds\}$ and $\bsk_u\in \nat^{|u|}$
  \begin{align*}
     \frac{1}{b^m}\sum_{n=0}^{b^m-1}\wal_{(\bsk_u,\bszero)}(\bsz_n) = \left\{ \begin{array}{ll}
     1 & \mathrm{if}\ (\bsk_u,\bszero)\in D^\perp_{\bsq,p} , \\
     0 & \mathrm{otherwise} . \\
     \end{array} \right.
  \end{align*}
Using this property we have
  \begin{align*}
    B_{\alpha,d,\bsgamma,(1)}(\bsq,p) & = \sum_{\emptyset \ne u\subseteq \{1:ds\}}\tilde{\gamma}_{w(u)}\sum_{\bsk_u\in \nat^{|u|}}\tilde{r}_{\alpha,d,(1)}(\bsk_u)\frac{1}{b^m}\sum_{n=0}^{b^m-1}\wal_{(\bsk_u,\bszero)}(\bsz_n) \nonumber \\
    & = \frac{1}{b^m}\sum_{n=0}^{b^m-1}\sum_{\emptyset \ne u\subseteq \{1:ds\}}\tilde{\gamma}_{w(u)}\prod_{j\in u}\sum_{k_j=1}^{\infty}\tilde{r}_{\alpha,d,(1)}(k_j) \wal_{k_j}(z_{n,j}) .
  \end{align*}
Following similar lines as in \cite[Section~2.2]{DP05}, we obtain for any $z\in [0,1)$
  \begin{align*}
    \sum_{k=1}^{\infty}\tilde{r}_{\alpha,d,(1)}(k) \wal_{k}(z) = \frac{1}{b^{\alpha/2}}\sum_{\xi=1}^{\infty}b^{-\min(\alpha,d)\xi}\sum_{k=b^{\xi-1}}^{b^\xi-1}\wal_{k}(z)=\phi_{\alpha,d,(1)}(z) .
  \end{align*}
Thus we have
  \begin{align*}
    B_{\alpha,d,\bsgamma,(1)}(\bsq,p) & = \frac{1}{b^m}\sum_{n=0}^{b^m-1}\sum_{\emptyset \ne u\subseteq \{1:ds\}}\tilde{\gamma}_{w(u)}\prod_{j\in u}\phi_{\alpha,d,(1)}(z_{n,j}) \\
    & = \frac{1}{b^m}\sum_{n=0}^{b^m-1}\sum_{\emptyset \ne v\subseteq \{1:s\}}\tilde{\gamma}_{v}\sum_{\substack{\emptyset \ne u\subseteq \{1:ds\}\\ w(u)=v}}\prod_{j\in u}\phi_{\alpha,d,(1)}(z_{n,j}) \\
    & = \frac{1}{b^m}\sum_{n=0}^{b^m-1}\sum_{\emptyset \ne v\subseteq \{1:s\}}\tilde{\gamma}_{v}\prod_{j\in v}\left[ -1+\prod_{l=1}^{d}\left( 1+\phi_{\alpha,d,(1)}(z_{n,(j-1)d+l})\right)\right] ,
  \end{align*}
where the last equality stems from the fact that at least one element of $\{(j-1)d+1:jd\}$ must be chosen in $u$ for any $j\in v$. Hence, the result follows.
\end{proof}

\subsection{Another bound on the worst-case error for $d\le \alpha$}
We derive another upper bound on the worst-case error, which applies to the case $d\le \alpha$ only. In the above subsection, we have used Lemma~\ref{lemma:weight} from \cite[Lemma~4]{Godxx} to cover both cases $d\le \alpha$ and $d>\alpha$. That lemma was obtained through the averaging argument on the digit interlacing function for non-negative integers, see the proof of \cite[Lemma~4]{Godxx}. Focusing on the case $d\le \alpha$, it is possible to obtain a tighter bound on the worst-case error without using the averaging argument.

As a counterpart of Lemma~\ref{lemma:weight}, we use the following lemma.

\begin{lemma}\label{lemma:weight2}
Let $d\le \alpha$. For $\emptyset \ne u\subseteq \{1:ds\}$ and $\bsk_u\in \nat^{|u|}$, we have
  \begin{align*}
    \mu_{\alpha}(\Ecal_d(\bsk_u,\bszero)) \ge \sum_{j\in u}\left[ d\mu_1(k_j)+j-d\lceil j/d\rceil \right] ,
  \end{align*}
where $\mu_1(k)$ is defined in (\ref{eq:dick_weight}).
\end{lemma}

\begin{proof}
For $j\in w(u)$, let $u_j:=u\cap\{(j-1)d+1:jd\}$. Then we have
  \begin{align*}
    \mu_{\alpha}(\Ecal_d(\bsk_u,\bszero)) = \sum_{j\in w(u)}\mu_{\alpha}(\Ecal_d(\bsk_{u_j},\bszero)_d) ,
  \end{align*}
where $(\bsk_{u_j},\bszero)_d$ is the vector in $\nat_0^{d}$ such that $k_l\in \nat$ for $l\in u_j$ and $k_l=0$ for $l\in \{(j-1)d+1:jd\}\setminus u_j$. Thus it suffices to prove that for any $j\in v(u)$
  \begin{align*}
    \mu_{\alpha}(\Ecal_d(\bsk_{u_j},\bszero)_d) \ge \sum_{l\in u_j}\left[ d\mu_1(k_l)+l-d\lceil l/d\rceil \right] .
  \end{align*}

Let us consider the definition of the weight $\mu_{\alpha}(k)$ as in (\ref{eq:dick_weight}). In order to evaluate the weight $\mu_{\alpha}(\Ecal_d(k_1,\ldots, k_d))$ for a given $k_1,\ldots,k_d\in \nat_0^d$ precisely, we need to reorder the summand in (\ref{eq:integer_interlace}) according to the value of $ad+r$. Instead we give a lower bound by only looking at the most significant digits for $k_1,\ldots,k_d$ based on their $b$-adic expansions.

Since we have $|u_j|\le d\le \alpha$, it follows that
  \begin{align*}
    \mu_{\alpha}(\Ecal_d(\bsk_{u_j},\bszero)_d) & \ge \sum_{l\in u_j}\left[ \left(\mu_{1}(k_{l})-1\right)d+l-(j-1)d \right] \\
    & = \sum_{l\in u_j}\left[ d\mu_{1}(k_{l})+l-jd \right] .
  \end{align*}
Considering $j=\lceil l/d\rceil$ for any $l\in u_j$, the proof is complete.
\end{proof}

Let us introduce the following notation
  \begin{align*}
     \tilde{r}_{d,j,(2)} (k) = \left\{ \begin{array}{ll}
     1                                    & \text{if}\; k=0, \\
     b^{-d\mu_1(k)-j+d\lceil j/d\rceil}   & \text{otherwise} .
     \end{array} \right.
  \end{align*}
Further we write $\tilde{r}_{d,u,(2)}(\bsk_u)=\prod_{j\in u}\tilde{r}_{d,j,(2)}(k_j)$. Then we have the following result, whose proof is almost the same as the proof of Theorem~\ref{theorem:bound1}.

\begin{theorem}\label{theorem:bound2}
Let $\Dcal_d(P_{b^m,ds}(\bsq,p))$ be an interlaced polynomial lattice point set of order $d$ and $D^\perp_{\bsq,p}$ be the dual polynomial lattice of $P_{b^m,ds}(\bsq,p)$. For $d\le \alpha$, the worst-case error for numerical integration in $\Wcal_{s,\alpha,\bsgamma}$ using $\Dcal_d(P_{b^m,ds}(\bsq,p))$ is bounded as
  \begin{align*}
     e(\Dcal_d(P_{b^m,ds}(\bsq,p)),\Wcal_{s,\alpha,\bsgamma}) \le \sum_{\emptyset \ne u\subseteq \{1:ds\}}\gamma_{w(u)}\sum_{\substack{\bsk_u\in \nat^{|u|}\\ (\bsk_u,\bszero)\in D^{\perp}(\bsq,p)}}\tilde{r}_{d,u,(2)}(\bsk_u) .
  \end{align*}
\end{theorem}

Since the upper bound given in Theorem~\ref{theorem:bound2} has a concise formula, we can also employ this bound as a quality criterion for searching for a good set of polynomials $\bsq$. For simplicity, we denote this bound by
  \begin{align}\label{eq:criterion2}
     B_{d,\bsgamma,(2)}(\bsq,p) := \sum_{\emptyset \ne u\subseteq \{1:ds\}}\gamma_{w(u)}\sum_{\substack{\bsk_u\in \nat^{|u|}\\ (\bsk_u,\bszero)\in D^{\perp}(\bsq,p)}}\tilde{r}_{d,u,(2)}(\bsk_u) .
  \end{align}
We note that $B_{d,\bsgamma,(2)}(\bsq,p)$ does not depend on $\alpha (\ge d)$, so that we have omitted $\alpha$ from the subscript. In the following, we provide a concise formula for $B_{d,\bsgamma,(2)}(\bsq,p)$. Since the proof is almost the same as that of Corollary~\ref{corollary:criterion}, we omit it.

\begin{corollary}\label{corollary:criterion2}
Let $\Dcal_d(P_{b^m,ds}(\bsq,p))$ be an interlaced polynomial lattice point set of order $d (\le \alpha)$ and $P_{b^m,ds}(\bsq,p)=\{\bsz_0,\ldots,\bsz_{b^m-1}\}$ be a polynomial lattice point set with generating vector $\bsq$ and modulus $p$. We have a concise formula for $B_{d,\bsgamma,(2)}(\bsq,p)$, which is given by
  \begin{align*}
     B_{d,\bsgamma,(2)}(\bsq,p) = \frac{1}{b^m}\sum_{n=0}^{b^m-1}\sum_{\emptyset \ne v\subseteq \{1:s\}}\gamma_{v}\prod_{j\in v}\left[ -1+\prod_{l=1}^{d}\left( 1+\frac{\phi_{d,(2)}(z_{n,(j-1)d+l})}{b^l}\right)\right] ,
  \end{align*}
where we define
  \begin{align*}
     \phi_{d,(2)}(z)=\frac{b^{d-1}(b-1-b^{(d-1)\lfloor \log_b z\rfloor}(b^d-1))}{b^{d-1}-1} ,
  \end{align*}
for any $z\in [0,1)$ in which we set $b^{\lfloor \log_b 0\rfloor}=0$.
\end{corollary}

\begin{remark}
For the first criterion $B_{\alpha,d,\bsgamma,(1)}(\bsq,p)$, every $d$ components have equal 'weights' in the innermost product. This is essentially due to the averaging argument applied to obtain Lemma~\ref{lemma:weight}, see the proof of \cite[Lemma~4]{Godxx}. For the second criterion $B_{d,\bsgamma,(2)}(\bsq,p)$ on the other hand, every $d$ components have different 'weights' $b^{-l}$ for $l=1,\ldots,d$ in the innermost product, indicating the relative importance of components with small $l$.
\end{remark}

\section{Component-by-component construction}\label{sec:cbc}
We employ $B_{\alpha,d,\bsgamma,(1)}(\bsq,p)$ or $B_{d,\bsgamma,(2)}(\bsq,p)$ as a quality criterion in a component-by-component (CBC) algorithm and investigate the CBC construction as a means of finding a good set of polynomials $\bsq$. We show that interlaced polynomial lattice rules thus constructed achieve the optimal rate of convergence of the worst-case error and discuss the dependence of the worst-case error on the dimension for the CBC construction.

\subsection{Construction algorithm and convergence rate}
In the CBC construction, we set $q_1=1$ without loss of generality and restrict $q_\tau$ for $2\le \tau\le ds$ such that $q_j\ne 0$ and $\deg(q_\tau)<m=\deg(p)$. In the following, we denote by $\Rcal_m$ the set of all non-zero polynomials over $\FF_b$ with degree less than $m$, that is,
  \begin{align*}
     \Rcal_m = \{q\in \FF_b[x]:\; \deg(q)<m\; \text{and}\; q\ne 0\}.
  \end{align*}
Assume that $\bsq_{\tau-1} =(q_1,\ldots,q_{\tau-1})\in (\Rcal_m)^{\tau-1}$ is given for some $1< \tau \le ds$. The idea of the CBC construction is to search for a polynomial $q_\tau\in \Rcal_m$ which minimizes $B_{\alpha,d,\bsgamma,(1)}((\bsq_{\tau-1},q_\tau),p)$ or $B_{d,\bsgamma,(2)}((\bsq_{\tau-1},q_\tau),p)$ with $\bsq_{\tau-1}$ unchanged. Thus we need to define $B_{\alpha,d,\bsgamma,(1)}(\bsq_\tau,p)$ or $B_{d,\bsgamma,(2)}(\bsq_\tau,p)$ for $1\le \tau\le ds$. Using the formula in Corollary~\ref{corollary:criterion}, we have
  \begin{align*}
     & B_{\alpha,d,\bsgamma,(1)}(\bsq_\tau,p) \\
     = & \frac{1}{b^m}\sum_{n=0}^{b^m-1}\sum_{\emptyset \ne v\subseteq \{1:j_0-1\}}\tilde{\gamma}_{v}\prod_{j\in v}\left[ -1+\prod_{l=1}^{d}\left( 1+\phi_{\alpha,d,(1)}(z_{n,(j-1)d+l})\right)\right] \\
     & + \frac{1}{b^m}\sum_{n=0}^{b^m-1}\sum_{v\subseteq \{1:j_0-1\}}\tilde{\gamma}_{v\cup \{j_0\}}\prod_{j\in v}\left[ -1+\prod_{l=1}^{d}\left( 1+\phi_{\alpha,d,(1)}(z_{n,(j-1)d+l})\right)\right] \\
     & \times \left[ -1+\prod_{l=1}^{d_0}\left( 1+\phi_{\alpha,d,(1)}(z_{n,(j_0-1)d+l})\right)\right] ,
  \end{align*}
where $j_0=\lceil \tau/d\rceil$ and $d_0=\tau-(j_0-1)d$. We also have a similar formula for $B_{d,\bsgamma,(2)}(\bsq_\tau,p)$, which is obvious from Corollary \ref{corollary:criterion2}. Now the CBC construction employing $B_{\alpha,d,\bsgamma,(1)}(\bsq,p)$ as a quality criterion can be summarized as follows.

\begin{algorithm}\label{algorithm:cbc}
For $s,m,\alpha,d\in \nat$ with $\min(\alpha,d)>1$ and $\bsgamma=(\gamma_v)_{v\subseteq \{1:s\}}$, do the following:
	\begin{enumerate}
		\item Choose an irreducible polynomial $p\in \FF_b[x]$ such that $\deg(p)=m$.
		\item Set $q_1=1$.
		\item For $\tau=2,\ldots, ds$, find $q_{\tau}$ which minimizes $B_{\alpha,d,\bsgamma,(1)}((\bsq_{\tau-1},\tilde{q}_\tau),p)$ as a function of $\tilde{q}_{\tau}\in \Rcal_m$.
	\end{enumerate}
\end{algorithm}
When we employ $B_{d,\bsgamma,(2)}(\bsq,p)$ as a quality criterion, we need to add one more condition $d\le \alpha$ and replace $B_{\alpha,d,\bsgamma,(1)}((\bsq_{\tau-1},\tilde{q}_\tau),p)$ with $B_{d,\bsgamma,(2)}((\bsq_{\tau-1},q_\tau),p)$ in Step 3.

We have the following theorems, which states that interlaced polynomial lattice rules constructed by Algorithm \ref{algorithm:cbc} achieve the optimal rate of convergence of the worst-case error. Since the proof follows along almost the same argument as the proofs of \cite[Theorem~2]{Godxx} and \cite[Theorem~1]{GDxx}, we omit it.

\begin{theorem}\label{theorem:cbc_bound1}
Let $s,m,\alpha,d\in \nat$, $\min(\alpha,d)>1$, and $\bsgamma=(\gamma_v)_{v\subseteq \{1:s\}}$ be given. Suppose $\bsq=(q_1,\ldots,q_{ds})$ is found by Algorithm~\ref{algorithm:cbc}. Then for any $\tau=1,\ldots,ds$ we have
  \begin{align*}
    & B_{\alpha,d,\bsgamma,(1)}(\bsq_\tau,p) \\
    \le & \frac{1}{(b^m-1)^{1/\lambda}}\left[ \sum_{\emptyset \ne v \subseteq \{1:j_0-1\}}\tilde{\gamma}_v^{\lambda}G_{\alpha,d,\lambda,d,(1)}^{|v|}+G_{\alpha,d,\lambda,d_0,(1)}\sum_{v \subseteq \{1:j_0-1\}}\tilde{\gamma}_{v\cup \{j_0\}}^{\lambda}G_{\alpha,d,\lambda,d,(1)}^{|v|} \right]^{1/\lambda} ,
  \end{align*}
for $1/\min(\alpha,d) < \lambda \le 1$, where we write $j_0=\lceil \tau/d\rceil$, $d_0=\tau-(j_0-1)d$ and
  \begin{align*}
    G_{\alpha,d,\lambda,a,(1)} = -1+(1+\tilde{G}_{\alpha,d,\lambda,(1)})^a ,
  \end{align*}
for $a=1,\ldots,d$, in which we define
  \begin{align*}
    \tilde{G}_{\alpha,d,\lambda,(1)} = \frac{1}{b^{\alpha \lambda/2}}\max\left\{ \left( \frac{b-1}{b^{\min(\alpha,d)}-b}\right)^{\lambda}, \frac{b-1}{b^{\lambda\min(\alpha,d)}-b}\right\} .
  \end{align*}
\end{theorem}

\begin{theorem}\label{theorem:cbc_bound2}
Let $s,m,\alpha,d\in \nat$, $1<d\le \alpha$, and $\bsgamma=(\gamma_v)_{v\subseteq \{1:s\}}$ be given. Suppose $\bsq=(q_1,\ldots,q_{ds})$ is found by Algorithm~\ref{algorithm:cbc}, where $B_{\alpha,d,\bsgamma,(1)}((\bsq_{\tau-1},\tilde{q}_\tau),p)$ is replaced with $B_{d,\bsgamma,(2)}((\bsq_{\tau-1},\tilde{q}_\tau),p)$. Then for any $\tau=1,\ldots,ds$ we have
  \begin{align*}
    & B_{d,\bsgamma,(2)}(\bsq_\tau,p) \\
    \le & \frac{1}{(b^m-1)^{1/\lambda}}\left[ \sum_{\emptyset \ne v \subseteq \{1:j_0-1\}}\gamma_v^{\lambda}G_{d,\lambda,d,(2)}^{|v|}+G_{d,\lambda,d_0,(2)}\sum_{v \subseteq \{1:j_0-1\}}\gamma_{v\cup \{j_0\}}^{\lambda}G_{d,\lambda,d,(2)}^{|v|} \right]^{1/\lambda} ,
  \end{align*}
for $1/\min(\alpha,d) < \lambda \le 1$, where we write $j_0=\lceil \tau/d\rceil$, $d_0=\tau-(j_0-1)d$ and
  \begin{align*}
    G_{d,\lambda,a,(2)} = -1+\prod_{l=1}^a \left[ 1+b^{\lambda(d-l)}\tilde{G}_{d,\lambda,(2)}\right] ,
  \end{align*}
for $a=1,\ldots,d$, in which we define
  \begin{align*}
    \tilde{G}_{d,\lambda,(2)} = \max\left\{ \left( \frac{b-1}{b^{\min(\alpha,d)}-b}\right)^{\lambda}, \frac{b-1}{b^{\lambda\min(\alpha,d)}-b}\right\} .
  \end{align*}
\end{theorem}

We compare the bounds on $B_{\alpha,d,\bsgamma,(1)}(\bsq,p)$ and $B_{d,\bsgamma,(2)}(\bsq,p)$, which are given in Theorems \ref{theorem:cbc_bound1} and \ref{theorem:cbc_bound2} respectively, with the bound on the worst-case error in $\Wcal_{s,\alpha,\bsgamma}$ for higher order polynomial lattice rules constructed component-by-component, which is given in \cite[Theorem~3.1]{BDGP11}. We consider $b=2$ and the unweighted case, that is, $\gamma_v=1$ for all $v\subseteq \{1:s\}$. In Figure~\ref{fig:bound_on_error}, we compare these three bounds on the worst-case error in $\Wcal_{s,\alpha,\bsgamma}$ for $s=2,4$ and $\alpha=d=2,3$ as a function of $\lambda$ such that $1/\min(\alpha,d)<\lambda \le 1$. In each graph, the number of points ranges from $2^8$ to $2^{24}$.

For $\alpha=d=2$, interlaced polynomial lattice rules based on $B_{\alpha,d,\bsgamma,(1)}(\bsq,p)$ are comparable to higher order polynomial lattice rules. For $\alpha=d=3$, however, interlaced polynomial lattice rules based on $B_{\alpha,d,\bsgamma,(1)}(\bsq,p)$ are inferior to higher order polynomial lattice rules. In every case, on the other hand, interlaced polynomial lattice rules based on $B_{d,\bsgamma,(2)}(\bsq,p)$ are superior to higher order polynomial lattice rules especially for large $\lambda$, indicating the usefulness of interlaced polynomial lattice rules.
\begin{figure}
\begin{center}
\includegraphics{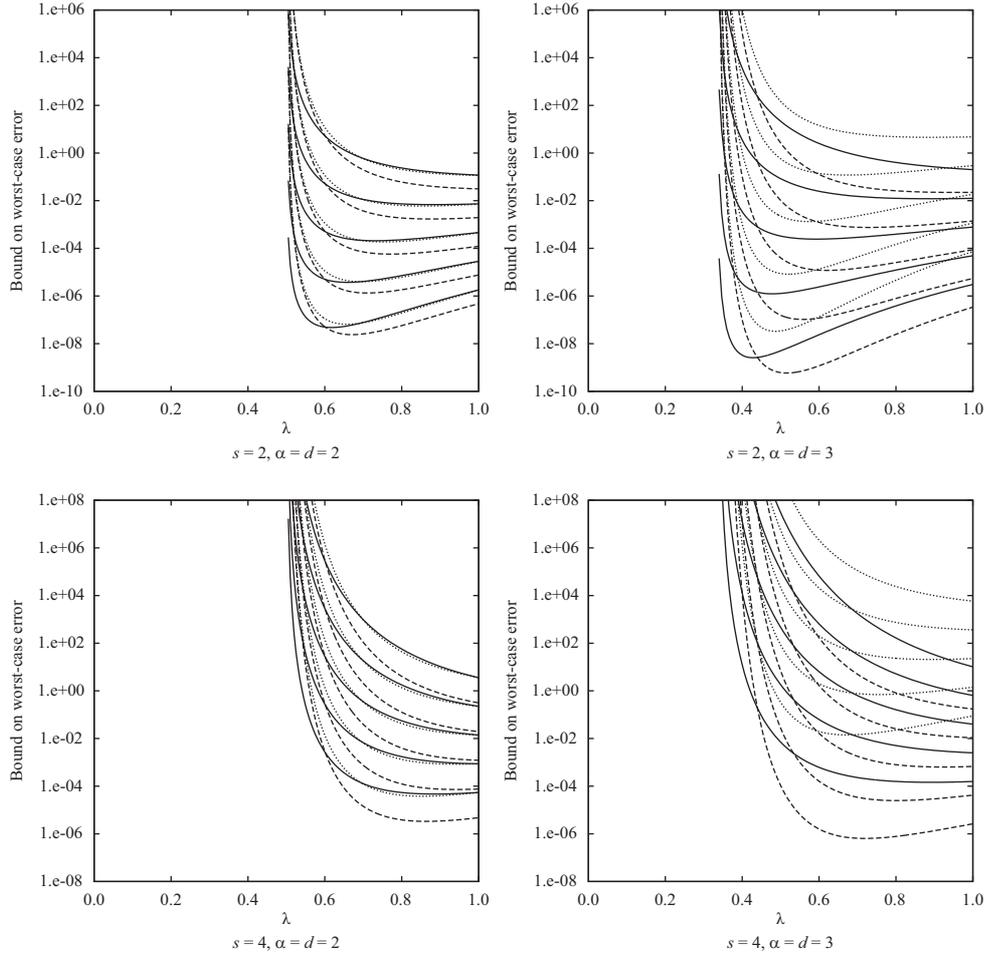}
\caption{Upper bounds on the worst-case error in $\Wcal_{s,\alpha,\bsgamma}$ of higher order polynomial lattice rules (lines) and upper bounds on $B_{\alpha,d,\bsgamma,(1)}(\bsq,p)$ and $B_{d,\bsgamma,(2)}(\bsq,p)$ of interlaced polynomial lattice rules (dots and dashed-lines, respectively) for $m=8,12,16,20,24$ with various choices of $s$ and $\alpha=d$.}
\label{fig:bound_on_error}
\end{center}
\end{figure}

In the following remark, we use the inequality, which states that for a sequence $(a_k)_{k\in \nat}$ of non-negative real numbers we have
  \begin{align}\label{eq:jensen}
    \left( \sum a_k\right)^{\lambda} \le \sum a^\lambda_k ,
  \end{align}
for any $\lambda$ with $0<\lambda\le 1$.

\begin{remark}\label{remark:cbc_propagation}
Let $s,m,\alpha,d\in \nat$ such that $1<\alpha\le d$, and $\bsgamma=(\gamma_v)_{v\subseteq \{1:s\}}$ be given. Suppose $\bsq=(q_1,\ldots,q_{ds})$ is found by Algorithm~\ref{algorithm:cbc}. From Theorem~\ref{theorem:cbc_bound1}, we know
  \begin{align*}
    B_{\alpha,d,\bsgamma,(1)}(\bsq,p) \le A_{\alpha,d,\bsgamma,\delta}b^{-\alpha m+\delta},
  \end{align*}
for any $\delta>0$. Let $\alpha'$ be an integer such that $\alpha\le \alpha'\le d$ and we consider $\gamma'_v=\gamma_v^{\alpha'/\alpha}$ for all $v\subseteq \{1:s\}$. We simply write $\bsgamma'=(\gamma'_v)_{v\subseteq \{1:s\}}$. Then we have
  \begin{align*}
    & B_{\alpha',d,\bsgamma',(1)}(\bsq,p) \\
    := & \sum_{\emptyset \ne u\subseteq \{1:ds\}}\gamma'_{w(u)}b^{\alpha'(2d-1)|w(u)|/2}\sum_{\substack{\bsk_u\in \nat^{|u|}\\ (\bsk_u,\bszero)\in D^\perp(\bsq,p)}}\tilde{r}_{\alpha',d,(1)}(\bsk_u) \\
    = & \sum_{\emptyset \ne u\subseteq \{1:ds\}}\left(\gamma_{w(u)}b^{\alpha(2d-1)|w(u)|/2}\right)^{\alpha'/\alpha}\sum_{\substack{\bsk_u\in \nat^{|u|}\\ (\bsk_u,\bszero)\in D^\perp(\bsq,p)}}(\tilde{r}_{\alpha,d,(1)}(\bsk_u))^{\alpha'/\alpha} \\
    \le & \left( \sum_{\emptyset \ne u\subseteq \{1:ds\}}\gamma_{w(u)}b^{\alpha(2d-1)|w(u)|/2}\sum_{\substack{\bsk_u\in \nat^{|u|}\\ (\bsk_u,\bszero)\in D^\perp(\bsq,p)}}\tilde{r}_{\alpha,d,(1)}(\bsk_u)\right) ^{\alpha'/\alpha} \\
    = & (B_{\alpha,d,\bsgamma,(1)}(\bsq,p))^{\alpha'/\alpha} \le A_{\alpha,d,\bsgamma,\delta}^{\alpha'/\alpha}b^{-\alpha' m+\delta \alpha'/\alpha} ,
  \end{align*}
for any $\delta>0$, where we have used (\ref{eq:jensen}) in the above inequality. This implies that interlaced polynomial lattice rules constructed by Algorithm~\ref{algorithm:cbc} for functions of smoothness $\alpha$ still achieve the optimal rate of convergence of the worst-case error for functions of smoothness $\alpha'$, as long as $\alpha\le \alpha'\le d$ holds.
\end{remark}

\begin{remark}\label{remark:fast_cbc}
In a similar way to \cite[Subsection~4,2]{Godxx} and \cite[Subsection~4,2]{GDxx}, it is possible to apply the fast CBC construction using the fast Fourier transform as in \cite{NC06a,NC06b} to our current setting. When $\gamma_u=\prod_{j\in u}\gamma_j$ for all $u\subseteq \{1:s\}$, for example, interlaced polynomial lattice rules of order $d$ can be constructed in $\Ocal(dsmb^m)$ operations using $\Ocal(b^m)$ memory. As mentioned in Subsection~\ref{ssec:ho_digital_nets}, the fast CBC construction of higher order polynomial lattice rules requires $\Ocal(dsmb^{dm})$ operations using $\Ocal(b^{dm})$ memory when $m'=dm$ in Definition~\ref{def:ho_polynomial_lattice}. This significant reduction in the construction cost enhances the practical usefulness of interlaced polynomial lattice rules.
\end{remark}

\subsection{Dependence of the error bounds on the dimension}\label{ssec:tractability_cbc}
We discuss the dependence of the worst-case error bounds on the dimension for the CBC construction. From Theorem~\ref{theorem:cbc_bound1}, we have the following corollary. It is straightforward to show a similar corollary for Theorem~\ref{theorem:cbc_bound2}.

\begin{corollary}\label{cor:tractability_cbc}
Let $s,m,\alpha,d\in \nat$, $\min(\alpha,d)>1$, and $\bsgamma=(\gamma_v)_{v\subseteq \{1:s\}}$ be given. Suppose $\bsq=(q_1,\ldots,q_{ds})$ is found by Algorithm~\ref{algorithm:cbc}. We define
  \begin{align*}
    A_{\lambda,q}:=\limsup_{s\to \infty}\left[\frac{1}{s^q}\sum_{\emptyset \ne v\subseteq \{1:s\}}\tilde{\gamma}_{v}^{\lambda}G_{\alpha,d,\lambda,d,(1)}^{|v|}\right] .
  \end{align*}
\begin{enumerate}
\item Assume $A_{\lambda,0}<\infty$ for some $1/\min(\alpha,d)<\lambda\le 1$. Then the worst-case error is bounded independently of the dimension.

\item Assume $A_{\lambda,q}<\infty$ for some $1/\min(\alpha,d)<\lambda\le 1$ and $q>0$. Then the worst-case error satisfies a bound which depends only polynomially on the dimension.
\end{enumerate}
\end{corollary}

\begin{proof}
Assume $A_{\lambda,q}<\infty$ for some $1/\min(\alpha,d)<\lambda\le 1$ and $q\ge 0$. From this assumption and Theorem~\ref{theorem:cbc_bound1} together with the fact that $B_{\alpha,d,\bsgamma,(1)}(\bsq,p)$ is a bound on $e(\Dcal_d(P_{b^m,ds}(\bsq,p)),\Wcal_{s,\alpha,\bsgamma})$, we have
  \begin{align*}
    e(\Dcal_d(P_{b^m,ds}(\bsq,p)),\Wcal_{s,\alpha,\bsgamma}) & \le \frac{1}{(b^m-1)^{1/\lambda}}\left[ \sum_{\emptyset \ne v \subseteq \{1:s\}}\tilde{\gamma}_v^{\lambda}G_{\alpha,d,\lambda,d,(1)}^{|v|}\right]^{1/\lambda} \\
    & \le \frac{(A_{\lambda, q}s^q)^{1/\lambda}}{(b^m-1)^{1/\lambda}} .
  \end{align*}
Thus, the worst-case error is bounded independently of the dimension when $q=0$, and satisfies a bound which depends only polynomially on the dimension when $q>0$.
\end{proof}

\section{Korobov construction}\label{sec:korobov}
As another means of finding a good set of polynomials $\bsq$, we investigate the Korobov construction. We show that interlaced polynomial lattice rules thus constructed achieve the optimal rate of convergence of the worst-case error and discuss the dependence of the worst-case error on the dimension for the Korobov construction.

\subsection{Construction algorithm and convergence rate}
In the Korobov construction, we set $q_1=1$ without loss of generality and restrict $q_\tau$ for $2\le \tau\le ds$ such that $q_j=q^{j-1} \pmod p$ for a common $q\in \Rcal_m$. That is, we only consider a generating vector of the form
  \begin{align*}
     (q_1,q_2,\ldots, q_{ds})=\psi_{ds}(q):=(1,q,\ldots, q^{ds-1}) \pmod p,
  \end{align*}
for $q\in \Rcal_{m}$. The idea of the Korobov construction is to search for a polynomial $q\in \Rcal_{m}$ which minimizes $B_{\alpha,d, \bsgamma,(1)}(\psi_{ds}(q),p)$ or $B_{d, \bsgamma,(2)}(\psi_{ds}(q),p)$. Thus the Korobov construction employing $B_{\alpha,d, \bsgamma,(1)}(\bsq,p)$ as a quality criterion can be summarized as follows.

\begin{algorithm}\label{algorithm:korobov}
For $s,m,\alpha,d\in \nat$, $\min(\alpha,d)>1$, and $\bsgamma=(\gamma_v)_{v\subseteq \{1:s\}}$, do the following:
	\begin{enumerate}
		\item Choose an irreducible polynomial $p\in \FF_b[x]$ such that $\deg(p)=m$.
		\item Find $q$ which minimizes $B_{\alpha,d, \bsgamma,(1)}(\psi_{ds}(\tilde{q}),p)$ as a function of $\tilde{q}\in \Rcal_m$.
	\end{enumerate}
\end{algorithm}
When we employ $B_{d,\bsgamma,(2)}(\bsq,p)$ as a quality criterion, we need to add one more condition $d\le \alpha$ and replace $B_{\alpha,d, \bsgamma,(1)}(\psi_{ds}(\tilde{q}),p)$ with $B_{d, \bsgamma,(2)}(\psi_{ds}(\tilde{q}),p)$ in Step 2.

We have the following theorems, which states that interlaced polynomial lattice rules thus constructed achieve the optimal rate of convergence of the worst-case error.

\begin{theorem}\label{theorem:korobov_bound1}
Let $s,m,\alpha,d\in \nat$, $\min(\alpha,d)>1$, and $\bsgamma=(\gamma_v)_{v\subseteq \{1:s\}}$ be given. Suppose that $\bsq=\psi_{ds}(q)$ is found by Algorithm~\ref{algorithm:korobov}. Then we have
  \begin{align*}
    B_{\alpha,d,\bsgamma,(1)}(\psi_{ds}(q),p) \le \frac{(ds)^{1/\lambda}}{(b^m-1)^{1/\lambda}}\left[ \sum_{\emptyset \ne v \subseteq \{1:s\}}\tilde{\gamma}_v^{\lambda}H_{\alpha,d,\lambda,(1)}^{|v|} \right]^{1/\lambda} ,
  \end{align*}
for $1/\min(\alpha,d) < \lambda \le 1$, where we write 
  \begin{align*}
   H_{\alpha,d,\lambda,(1)} = -1+(1+\tilde{H}_{\alpha,d,\lambda,(1)})^d ,
  \end{align*}
in which we define
  \begin{align*}
   \tilde{H}_{\alpha,d,\lambda,(1)} = \frac{1}{b^{\alpha \lambda/2}}\cdot\frac{b-1}{b^{\lambda\min(\alpha,d)}-b} ,
  \end{align*}
\end{theorem}

\begin{theorem}\label{theorem:korobov_bound2}
Let $s,m,\alpha,d\in \nat$, $1<d\le \alpha$, and $\bsgamma=(\gamma_v)_{v\subseteq \{1:s\}}$ be given. Suppose that $\bsq=\psi_{ds}(q)$ is found by Algorithm~\ref{algorithm:korobov}, in which $B_{\alpha,d,\bsgamma,(1)}(\psi_{ds}(\tilde{q}),p)$ is replaced with $B_{d,\bsgamma,(2)}(\psi_{ds}(\tilde{q}),p)$. Then we have
  \begin{align*}
    B_{d,\bsgamma,(2)}(\psi_{ds}(q),p) \le \frac{(ds)^{1/\lambda}}{(b^m-1)^{1/\lambda}}\left[ \sum_{\emptyset \ne v \subseteq \{1:s\}}\tilde{\gamma}_v^{\lambda}H_{d,\lambda,(2)}^{|v|} \right]^{1/\lambda} ,
  \end{align*}
for $1/\min(\alpha,d) < \lambda \le 1$, where we write 
  \begin{align*}
   H_{d,\lambda,(2)} = -1+\prod_{l=1}^d\left[ 1+b^{\lambda(d-l)}\tilde{H}_{d,\lambda,(2)}\right] ,
  \end{align*}
in which we define
  \begin{align*}
   \tilde{H}_{d,\lambda,(2)} = \frac{b-1}{b^{\lambda\min(\alpha,d)}-b} .
  \end{align*}
\end{theorem}

The proof of these theorems follows along similar lines as the proof of \cite[Theorem~4.7]{DKPS05}, while additional treatment is required for interlacing components. In the following, we only focus on Theorem~\ref{theorem:korobov_bound1} and give its proof. In the proof, we shall use inequality (\ref{eq:jensen}) and the following lemma.

\begin{lemma}\label{lemma:sum_weight}
For any $\lambda$ with $1/\min(\alpha,d)<\lambda\le 1$ and any $m\in \nat_0$, we have
  \begin{align*}
     \sum_{k=1}^{\infty}\tilde{r}_{\alpha,d,(1)}^{\lambda}(b^m k)=\frac{\tilde{H}_{\alpha,d,\lambda,(1)}}{b^{\lambda \min(\alpha,d)m}} ,
  \end{align*}
where $\tilde{H}_{\alpha,d,\lambda,(1)}$ is defined as in Theorem~\ref{theorem:korobov_bound1}.
\end{lemma}

\begin{proof}
From the definition of $\tilde{r}_{\alpha,d,(1)}(k)$, we have 
  \begin{align*}
     \tilde{r}_{\alpha,d,(1)}(b^m k) = b^{-\min(\alpha,d)m} \tilde{r}_{\alpha,d,(1)}(k) .
  \end{align*}
Thus we obtain
  \begin{align*}
     \sum_{k=1}^{\infty}\tilde{r}_{\alpha,d,(1)}^{\lambda}(b^m k) & = \frac{1}{b^{\lambda\min(\alpha,d)m}}\sum_{k=1}^{\infty}\tilde{r}_{\alpha,d,(1)}^{\lambda}(k) \\
     & = \frac{1}{b^{\lambda\min(\alpha,d)m}}\sum_{\xi=1}^{\infty}\sum_{k=b^{\xi-1}}^{b^\xi-1}b^{-\lambda\min(\alpha,d)\xi-\alpha\lambda/2} \\
     & = \frac{1}{b^{\lambda\min(\alpha,d)m+\alpha \lambda /2}}\sum_{\xi=1}^{\infty}(b^\xi-b^{\xi-1})b^{-\lambda\min(\alpha,d)\xi} \\
     & = \frac{1}{b^{\lambda\min(\alpha,d)m+\alpha \lambda /2}}\cdot \frac{b-1}{b^{\lambda\min(\alpha,d)}-b} .
  \end{align*}
Hence the result follows.
\end{proof}

\begin{proof}[Proof of Theorem~\ref{theorem:korobov_bound1}]
Since $B_{\alpha,d, \bsgamma,(1)}(\psi_{ds}(q),p)\le B_{\alpha,d, \bsgamma,(1)}(\psi_{ds}(\tilde{q}),p)$ for all $\tilde{q}\in \Rcal_m$, $B_{\alpha,d, \bsgamma,(1)}^\lambda(\psi_{ds}(q),p)$ has to be less than or equal to the average of $B_{\alpha,d, \bsgamma,(1)}^\lambda(\psi_{ds}(\tilde{q}),p)$ over $\tilde{q}\in \Rcal_m$ for $1/\min(\alpha,d)<\lambda\le 1$. Thus, we have
  \begin{align*}
    & B_{\alpha,d,\bsgamma,(1)}^\lambda(\psi_{ds}(q),p) \\
    \le & \frac{1}{b^m-1}\sum_{\tilde{q}\in \Rcal_m}B_{\alpha,d,\bsgamma,(1)}^\lambda(\psi_{ds}(\tilde{q}),p) \\
    \le & \frac{1}{b^m-1}\sum_{\tilde{q}\in \Rcal_m}\sum_{\emptyset \ne u\subseteq \{1:ds\}}\tilde{\gamma}_{w(u)}^\lambda \sum_{\substack{\bsk_u\in \nat^{|u|} \\ (\bsk_u,\bszero)\in D^\perp(\psi_{ds}(\tilde{q}),p)}}\tilde{r}_{\alpha,d,(1)}^\lambda(\bsk_u) \\
    = & \sum_{\emptyset \ne u\subseteq \{1:ds\}}\tilde{\gamma}_{w(u)}^\lambda\sum_{\bsk_u\in \nat^{|u|}}\tilde{r}_{\alpha,d,(1)}^\lambda(\bsk_u) \frac{1}{b^m-1}\sum_{\substack{\tilde{q}\in \Rcal_m\\ \rtr_m(\bsk_u)\cdot \psi_u(\tilde{q})\equiv 0\pmod p}}1 ,
  \end{align*}
for $1/\min(\alpha,d)<\lambda\le 1$, where we have used (\ref{eq:jensen}) in the second inequality and introduced the notation $\psi_u(\tilde{q}) = (\tilde{q}^{j-1})_{j\in u} \pmod p$.

We now follow along an argument as in the proof of \cite[Theorem~4.7]{DKPS05} to count the number of $\tilde{q}\in \Rcal_m$ satisfying $\rtr_m(\bsk_u)\cdot \psi_u(\tilde{q})\equiv 0\pmod p$ for a given $\bsk_u\in \nat^{|u|}$. First, we recall that for an irreducible polynomial $p\in \FF_b[x]$ with $\deg(p)=m$ and a non-zero $(\bar{k}_1,\ldots,\bar{k}_{ds})\in (\Rcal_m)^{ds}$, the congruence
  \begin{align*}
    \bar{k}_1+\bar{k}_2\tilde{q}+\cdots +\bar{k}_{ds}\tilde{q}^{ds-1}\equiv 0\pmod p
  \end{align*}
has at most $ds-1$ solutions $\tilde{q}\in \Rcal_m$.

For $\emptyset \ne u\subseteq \{1:ds\}$, we consider two cases:
\begin{enumerate}
\item For all $j\in u$, let $k_j=b^m l_j$ be such that $l_j\in \nat$. In this case, we have $\rtr_m(k_j)\equiv 0\pmod p$ for all $j\in u$. Thus we have
  \begin{align*}
    \sum_{\substack{\tilde{q}\in \Rcal_m\\ \rtr_m(\bsk_u)\cdot \psi_u(\tilde{q})\equiv 0\pmod p}}1=b^m-1.
  \end{align*}
\item Let $u^*$ be any non-empty subset of $u$. For all $j\in u\setminus u^*$, let $k_j=b^m l_j$ be such that $l_j\in \nat$. Further, for all $j\in u^*$, let $k_j=b^m l_j+l^*_j$ be such that $l_j\in \nat$ and $1\le l^*_j<b^m$. In this case, we have $\rtr_m(k_j)\equiv 0\pmod p$ for all $j\in u\setminus u^*$ and $\rtr_m(k_j)\not\equiv 0\pmod p$ for all $j\in u^*$. Thus we have
  \begin{align*}
    \sum_{\substack{\tilde{q}\in \Rcal_m\\ \rtr_m(\bsk_u)\cdot \psi_u(\tilde{q})\equiv 0\pmod p}}1\le ds-1.
  \end{align*}
\end{enumerate}

Now we obtain
  \begin{align*}
    & B_{\alpha,d,\bsgamma}^\lambda(\psi_{ds}(q),p) \\
    \le & \sum_{\emptyset \ne u\subseteq \{1:ds\}}\tilde{\gamma}_{w(u)}^\lambda\sum_{\bsl_u\in \nat^{|u|}}\tilde{r}_{\alpha,d,(1)}^\lambda(b^m\bsl_u) \\
    & + \frac{ds-1}{b^m-1}\sum_{\emptyset \ne u\subseteq \{1:ds\}}\tilde{\gamma}_{w(u)}^\lambda\left[\sum_{\bsk_u\in \nat^{|u|}}\tilde{r}_{\alpha,d,(1)}^\lambda(\bsk_u)-\sum_{\bsl_u\in \nat^{|u|}}\tilde{r}_{\alpha,d,(1)}^\lambda(b^m\bsl_u)\right] \\
    = & \sum_{\emptyset \ne u\subseteq \{1:ds\}}\tilde{\gamma}_{w(u)}^\lambda\left(\frac{\tilde{H}_{\alpha,d,\lambda,(1)}}{b^{\lambda \min(\alpha,d)m}}\right)^{|u|} \\
    & + \frac{ds-1}{b^m-1}\sum_{\emptyset \ne u\subseteq \{1:ds\}}\tilde{\gamma}_{w(u)}^\lambda\left[ \tilde{H}_{\alpha,d,\lambda,(1)}^{|u|}-\left( \frac{\tilde{H}_{\alpha,d,\lambda,(1)}}{b^{\lambda \min(\alpha,d)m}}\right)^{|u|}\right] \\
    \le & \frac{1}{b^m-1}\sum_{\emptyset \ne u\subseteq \{1:ds\}}\tilde{\gamma}_{w(u)}^\lambda \tilde{H}_{\alpha,d,\lambda,(1)}^{|u|}+\frac{ds-1}{b^m-1}\sum_{\emptyset \ne u\subseteq \{1:ds\}}\tilde{\gamma}_{w(u)}^\lambda \tilde{H}_{\alpha,d,\lambda,(1)}^{|u|} \\
    = & \frac{ds}{b^m-1}\sum_{\emptyset \ne u\subseteq \{1:ds\}}\tilde{\gamma}_{w(u)}^\lambda \tilde{H}_{\alpha,d,\lambda,(1)}^{|u|} \\
    = & \frac{ds}{b^m-1}\sum_{\emptyset \ne v\subseteq \{1:s\}}\tilde{\gamma}_{v}^\lambda \sum_{\substack{\emptyset \ne u\subseteq \{1:ds\}\\ w(u)=v}}\tilde{H}_{\alpha,d,\lambda,(1)}^{|u|} \\
    = & \frac{ds}{b^m-1}\sum_{\emptyset \ne v\subseteq \{1:s\}}\tilde{\gamma}_{v}^\lambda \left[ -1+(1+\tilde{H}_{\alpha,d,\lambda,(1)})^d\right]^{|v|} ,
  \end{align*}
where we have used Lemma~\ref{lemma:sum_weight} in the first equality. Hence the result follows.
\end{proof}

\begin{remark}\label{remark:korobov_propagation}
Let $s,m,\alpha,d\in \nat$, $\min(\alpha,d)>1$, and $\bsgamma=(\gamma_v)_{v\subseteq \{1:s\}}$ be given. Suppose $\bsq=(q_1,\ldots,q_{ds})$ is found by Algorithm~\ref{algorithm:korobov}. From Theorem \ref{theorem:korobov_bound1}, we know
  \begin{align*}
    B_{\alpha,d,\bsgamma,(1)}(\bsq,p) \le \tilde{A}_{\alpha,d,\bsgamma,\delta}b^{-\alpha m+\delta},
  \end{align*}
for any $\delta>0$. As in Remark~\ref{remark:cbc_propagation}, it can be concluded that interlaced polynomial lattice rules constructed by Algorithm~\ref{algorithm:korobov} for functions of smoothness $\alpha$ still achieve the optimal rate of convergence of the worst-case error for functions of smoothness $\alpha'$, as long as $\alpha\le \alpha'\le d$ holds.
\end{remark}

\begin{remark}\label{remark:korobov_comparison}
For the Korobov construction of higher order polynomial lattice rules, we need to search for $q$ not from $\Rcal_{m}$ but from $\Rcal_{m'}$ such that it minimizes the worst-case error, see \cite{BDGP11}. It implies that $q$ has the exponentially larger number of candidates when $m'=dm$ as compared to Algorithm~\ref{algorithm:korobov}. Therefore, interlaced polynomial lattice rules can significantly reduce the number of candidates while still achieving the optimal rate of convergence of the worst-case error as shown in Theorems~\ref{theorem:korobov_bound1} and~\ref{theorem:korobov_bound2}.
\end{remark}

\subsection{Dependence of the error bounds on the dimension}
As in Subsection~\ref{ssec:tractability_cbc}, here we discuss the dependence of the worst-case error bounds on the dimension for the Korobov construction. From Theorem~\ref{theorem:korobov_bound1}, we have the following corollary. It is straightforward to show a similar corollary for Theorem~\ref{theorem:korobov_bound2}. Since the proof is almost the same as that of Corollary~\ref{cor:tractability_cbc}, we omit it.

\begin{corollary}
Let $s,m,\alpha,d\in \nat$, $\min(\alpha,d)>1$, and $\bsgamma=(\gamma_v)_{v\subseteq \{1:s\}}$ be given. Suppose $\bsq=(q_1,\ldots,q_{ds})$ is found by Algorithm~\ref{algorithm:korobov}. We define
  \begin{align*}
    \tilde{A}_{\lambda,q}:=\limsup_{s\to \infty}\left[\frac{1}{s^q}\sum_{\emptyset \ne v\subseteq \{1:s\}}\tilde{\gamma}_{v}^{\lambda}H_{\alpha,d,\lambda,(1)}^{|v|}\right] .
  \end{align*}
Assume $\tilde{A}_{\lambda,q}<\infty$ for some $1/\min(\alpha,d)<\lambda\le 1$ and $q\ge 0$. Then the worst-case error satisfies a bound which depends only polynomially on the dimension.
\end{corollary}

\section*{Acknowledgments}
This work was supported by Grant-in-Aid for JSPS Fellows No.24-4020. The author would like to thank Josef Dick and two anonymous referees for their valuable comments.

\end{document}